\documentclass{amsart}
\usepackage{enumerate}
\usepackage{amssymb}
\usepackage{mathtools}
\usepackage{mathrsfs}
\usepackage{color}
\usepackage[
colorlinks=true,
linkcolor=blue,
citecolor=blue,
urlcolor=blue]{hyperref}

\numberwithin{equation}{section}
\newtheorem{thm}{Theorem}[section]
\newtheorem{lem}[thm]{Lemma}
\newtheorem{defi}[thm]{Definition}
\newtheorem{rem}[thm]{Remark}
\newtheorem{cor}[thm]{Corollary}

\newtheorem{prop}[thm]{Proposition}

\newcommand{\N}{\mathbb N}

\title[Uniformly convex stalks and $\ell_{\infty}(E)$]{Uniformly convex stalks and Kucher's $\ell_{\infty}(E)$ problem}

\author[H. Li]{Hui Li}
 \address{School of Mathematics  and Statistics, Central South University, HNP-LAMA, Changsha 410075, China and  Department of Mathematics}
 \email{lihuicsu@csu.edu.cn}	
	
\author[S. Luo]{Sijie Luo}
\address{School of Mathematics and Statistics, Central South University, HNP-LAMA, Changsha 410075, China}
\email{sijieluo@csu.edu.cn}
	
\author[L. Wu]{Lian Wu}
\address{School of Mathematics and Statistics, Central South University, HNP-LAMA, Changsha 410075, China}
\email{wulian@csu.edu.cn}	

\author[D. Zhou]{Dejian Zhou}
\address{School of Mathematics and Statistics, Central South University, HNP-LAMA, Changsha 410075, China}
\email{zhoudejian@csu.edu.cn}
\date{}

\begin{document}
\begin{abstract}
Kucher asked whether $\ell_{\infty}(E)$ is a Grothendieck space for every super-reflexive Banach space $E$. We answer this question affirmatively by proving the stronger result that $\ell_{\infty}(E)$ has Pełczyński's property $(V)$. Combined with the natural dual representation of $\ell_{\infty}(E)$, this yields that every non-weakly compact operator on $\ell_{\infty}(E)$ fixes a copy of $\ell_{\infty}$.
The main ingredient is a general property-$(V)$ theorem for function modules
over zero-dimensional compact Hausdorff spaces with stalks admitting a common modulus of uniform convexity. The proof is based on Gierz's integral representation of functionals and the weak compactness of the associated canonical measures. In the case of $\ell_{\infty}(E)$, the relevant stalks are ultrapowers of $E$, and super-reflexivity provides the uniform geometric control required by the theorem. We also derive several consequences for operators on $\ell_{\infty}(E)$, including weak compactness into separable spaces, reflexivity of separable quotients, and compactness of operators into Schur spaces.
\end{abstract}
\maketitle

\section{Introduction}\label{Introduction}

A Banach space $X$ is called \emph{Grothendieck} if every weak-$*$ convergent sequence in $X^{*}$ is weakly convergent. Grothendieck's theorem for $\ell_{\infty}$ provides the fundamental nonreflexive example \cite{Grothendieck1953}. The property admits an important operator characterization. Namely, $X$ is Grothendieck if and only if every operator from $X$ into a separable Banach space is weakly compact. Equivalently, it is enough to verify weak compactness only for operators from $X$ into $c_{0}$ \cite{Diestel1984,GonzalezKania2021}.

The stability of the Grothendieck property under direct sums is well understood when the outer sequence space is reflexive. In particular, if a reflexive Banach space with an unconditional basis is used as the outer space of a direct sum, then the resulting space is Grothendieck precisely when each summand is Grothendieck, see  \cite[Section~0.21]{GonzalezKania2021}. The endpoint case
\[
\ell_{\infty}(E) = \left( \bigoplus_{n=1}^{\infty} E \right)_{\ell_{\infty}}
\]
is fundamentally different, since the outer space is nonreflexive and the standard arguments for reflexive direct sums no longer apply. Earlier positive results at this endpoint relied on additional order structure. Leung proved the result for certain countably order-complete Banach lattices \cite{Leung1988}, and Leung--R\"{a}biger further investigated the role of property $(V)$ in $\ell_{\infty}$-sums \cite{LeungRabiger1990}. Kucher later showed that $\ell_{\infty}(E)$ is Grothendieck when $E$ is a $B$-convex Banach lattice and asked whether the lattice assumption could be removed, namely whether the same conclusion holds for every super-reflexive Banach space $E$ \cite{Kucher1999}. This question was subsequently listed as Problem~29 by González and Kania \cite{GonzalezKania2021}.

Our approach is based on Pełczyński's property $(V)$, which provides a natural connection between unconditional convergence and weak compactness. We refer to Section~\ref{Preliminaries} for the relevant definitions and background. Pełczyński introduced this property and proved that every scalar $C(K)$-space has property $(V)$ \cite{Pelczynski1962}. For a Banach space isomorphic to a complemented subspace of a dual space, property $(V)$ has a stronger consequence. Every non-weakly compact operator fixes a copy of $\ell_{\infty}$, and in particular the space is Grothendieck. The precise statement is recalled in Proposition~\ref{prop:dual-bridge}. In the sequel, we assume that $E$ is a super-reflexive Banach space. The canonical dual representation
\[
\ell_{\infty}(E) = \ell_{\infty}(E^{**}) \cong (\ell_{1}(E^{*}))^{*}
\]
places $\ell_{\infty}(E)$ precisely in this setting. We therefore seek to establish property $(V)$ directly.

\begin{thm}\label{thm:main}
Let $E$ be a complex super-reflexive Banach space. Then $\ell_{\infty}(E)$ has Pe{\l}czy\'{n}ski's property $(V)$ and is a Grothendieck space.
\end{thm}

Theorem~\ref{thm:main} yields more than the Grothendieck property. It shows that every non-weakly compact operator defined on $\ell_{\infty}(E)$ fixes a copy of $\ell_{\infty}$. The relevant structure of $\ell_{\infty}(E)$ is its natural $C(\beta\mathbb{N})$-module structure. Unlike the scalar identity $\ell_{\infty} = C(\beta\mathbb{N}),$ the vector-valued space $\ell_{\infty}(E)$ cannot, in general, be viewed as a constant-stalk function space when $E$ is infinite-dimensional. The stalks arising from this representation are ultrapowers of $E$, rather than copies of $E$ itself. Super-reflexivity is therefore the natural assumption, as an equivalent uniformly convex renorming of $E$ yields a common modulus of convexity for all of its ultrapowers. This observation leads to our main structural theorem for function modules, which reveals the underlying principle behind the $\ell_{\infty}(E)$ case.

For a nonzero Banach space $G$, its modulus of convexity is defined by
\begin{equation}\label{eq:modulus-definition}
\delta_G(\varepsilon)=\inf\left\{1-\left\|\frac{u+v}{2}\right\|: \|u\|=\|v\|=1,\ \|u-v\|\ge\varepsilon\right\},\qquad 0<\varepsilon\le2.
\end{equation}

\begin{thm}\label{thm:module}
Let $K$ be a zero-dimensional compact Hausdorff space and let $Z$ be a uniformly separable complex function module over $K$. Suppose that there is a function $\delta:(0,2]\to(0,1]$ such that the stalks admit a common uniform convexity estimate
\begin{equation}\label{eq:common-modulus}
\delta_{Z_t}(\varepsilon)\ge\delta(\varepsilon)>0,
\qquad t\in K,\ 0<\varepsilon\le2.
\end{equation}
Then $Z$ has property $(V)$. 
\end{thm}

The theorem shows that the property $(V)$ is determined by the uniform geometry of the stalks rather than by the particular form of the module. We now apply it to $\ell_{\infty}(E)$. Its canonical $C(\beta\mathbb{N})$-module structure has ultrapower stalks, and super-reflexivity provides the common modulus of uniform convexity required by the theorem.
Therefore, $\ell_{\infty}(E)$ has property $(V)$, which resolves Kucher's problem and yields the stronger operator-theoretic conclusion stated in Theorem~\ref{thm:main}. The remaining consequences, including the structure of operators and separable quotients, are discussed in Corollary \ref{remaining}.

The rest of the paper is organized as follows. 
In Section~\ref{Preliminaries} , we provide the necessary background and notation. 
Section~\ref{measure-integrability} establishes the weak compactness of the canonical scalar measures and the uniform integrability of the corresponding Radon--Nikod\'ym derivatives. In Section~\ref{proof of Theorem 1.2}, we combine these compactness arguments with the common modulus of uniform convexity of the stalks to establish the general function-module theorem (i.e. Theorem~\ref{thm:module}). Section~\ref{proof of Theorem 1.1} identifies the stalks of the natural $C(\beta\mathbb{N})$-module structure of $\ell_\infty(E)$, carries out the required reduction to uniformly separable submodules, and applies the function-module theorem to prove Theorem~\ref{thm:main}. It also derives the resulting consequences for operators and separable quotients.

\textbf{Scalar convention} Gierz's 1987 construction is stated for real functionals. We therefore prove the function-module theorem over $\mathbb{R}$. The complex case then follows by passing to the underlying real structure. Indeed, every continuous real functional on a complex Banach space is the real part of a continuous complex-linear functional. Hence the real and complex weak topologies agree, as do weak compactness, weak unconditional convergence, unconditional convergence, and the Grothendieck property. Moreover, uniform convexity and super-reflexivity are invariant under this realification, and $\ell_{\infty}(E)_{\mathbb{R}} = \ell_{\infty}(E_{\mathbb{R}})$. Thus a complex-linear operator may be treated as a real-linear operator in the argument.

In the sequel, we write $B_X$ for the closed unit ball of a Banach space $X$. The space $M(K)$ consists of finite regular scalar Borel measures on a compact Hausdorff space $K$, with the total variation norm, thus $M(K)=C(K)^{*}$. Weak compactness in $M(K)$ always refers to $\sigma(M(K),M(K)^{*})$, rather than to the weak-$*$ topology induced by $C(K)$. 

\section{Preliminaries}\label{Preliminaries}


A series $\sum_{n} x_{n}$ in a Banach space $X$ is \emph{weakly unconditionally convergent} (w.u.c.) if $\sum_{n}|x^{*}(x_{n})|<\infty$ for every $x^{*}\in X^{*}$. By the uniform boundedness principle, the weak unconditional convergence of $\sum_{n} x_{n}$ admits the following equivalent characterization:
\begin{equation}\label{eq:wuc-test}
\sum_n x_n\text{ is w.u.c.}
\quad\Longleftrightarrow\quad
\sup_{\substack{F\subset\N\\F\text{ finite}}}
\ \sup_{|\alpha_n|\le1}
\left\|\sum_{n\in F}\alpha_nx_n\right\|<\infty.
\end{equation}
An operator $T : X \to Y$ is \emph{unconditionally converging} if $\sum_{n} Tx_{n}$ converges unconditionally whenever $\sum_{n} x_{n}$ is w.u.c. The space $X$ has \emph{Pe{\l}czy\'{n}ski's property $(V)$} if every unconditionally converging operator from $X$ to an arbitrary Banach space is weakly compact, see \cite{Pelczynski1962}. We also need the following result established by Gonz\'{a}lez and Kania.

\begin{prop}[{\cite[Proposition~3.2.5]{GonzalezKania2021}}]\label{prop:dual-bridge}
If a Banach space $X$ has property $(V)$ and is isomorphic to a complemented subspace of a dual space. Then every non-weakly compact operator from $X$ is an isomorphism on a subspace isomorphic to $\ell_{\infty}$. In particular, $X$ is Grothendieck.
\end{prop}

Property $(V)$ and the Grothendieck property are isomorphic invariants. We will therefore renorm spaces without further comment.
A compact Hausdorff space is \emph{zero-dimensional} if its clopen sets form a base. In such a space, every compact subset of an open set is contained in a clopen subset of that open set. Also, locally constant scalar functions are uniformly dense in $C(K)$. We recall the relevant definitions and the precise part of Gierz’s theorem needed below, referring to \cite{Gierz1982,Gierz1987} for further details. 

\begin{defi}[Gierz]\label{def:module}
Let $K$ be a compact Hausdorff space and let $(Z_{t})_{t\in K}$ be Banach spaces. A \emph{function module} over $K$ is a closed linear subspace
\[
Z\subset\prod_{t\in K}Z_{t},\qquad
\|z\|=\sup_{t\in K}\|z(t)\|,
\]
with the following properties:
\begin{enumerate}[\rm (i)]
\item $t\mapsto\|z(t)\|$ is upper semicontinuous for every $z\in Z$;
\item evaluation $z\mapsto z(t)$ maps $Z$ onto $Z_{t}$ for every $t\in K$;
\item $fz\in Z$ for $f\in C(K)$ and $z\in Z$, where $(fz)(t)=f(t)z(t)$.
\end{enumerate}
For a given $t\in K$, the Banach space $Z_{t}$ is called the stalk over $t$. 
\end{defi}

The module is \emph{uniformly separable} if there is a countable set $D \subset Z$ such that $\{d(t) : d \in D\}$ is dense in every $Z_{t}$. Equivalently, one may use a separable linear subspace of $Z$ with stalkwise dense evaluations. Over $\mathbb{C}$, the countable core can be taken over $\mathbb{Q} + i\mathbb{Q}$.
For a closed set $A \subset K$, put
\[
N_{A} = \{z \in Z : z|_{A} = 0\}.
\]
Gierz proves that $N_{A}$ is an $M$-ideal and constructs associated $L$-projections on $Z^{*}$. The construction extends through a complete Boolean algebra containing all Borel subsets of $K$. The following formulation collects the relevant results from Sections 2--3 and Theorem 3.2 of \cite{Gierz1987}.

\begin{thm}[Gierz]\label{Gierz}
Let $Z$ be a uniformly separable real function module over a compact Hausdorff space $K$. For every $\phi \in Z^{*}$ there are:

\begin{enumerate}[\rm (i)]
\item $L$-projections $p_{M}$ associated with Borel subsets $M \subset K$;

\item a finite positive regular Borel measure
\begin{equation}\label{Gierz-measure}
\mu_{\phi}(M) = \|p_{M}\phi\|, \qquad \mu_{\phi}(K) = \|\phi\|; 
\end{equation}

\item stalk functionals $\xi_{\phi}(t) \in Z_{t}^{*}$ with $\|\xi_{\phi}(t)\| \leq 1$, such that $t \mapsto \xi_{\phi}(t)(z(t))$ is Borel measurable for every $z \in Z$ and
\begin{equation} \label{Gierz-representation}
\phi(z) = \int_{K} \xi_{\phi}(t)(z(t)) \, d\mu_{\phi}(t) \qquad (z \in Z). 
\end{equation}
\end{enumerate}
The measure in \eqref{Gierz-representation} is the canonical measure constructed in \eqref{Gierz-measure}.
\end{thm}

We call $\mu_{\phi}$ the canonical scalar measure associated with
$\phi$. For later use, we record the following identification of the projection associated with a clopen set.

\begin{lem}\label{lem clopen}
Let $C\subset K$ be clopen and define $P_{C}z=\mathbf{1}_{C}z$ for $z\in Z$. Then, for every $\phi\in Z^{*}$,
\begin{equation}\label{clopen-identification}
p_{C}\phi=P_{C}^{*}\phi=\phi\circ P_{C},
\qquad
\mu_{\phi}(C)=\|\phi\circ P_{C}\|.
\end{equation}
\end{lem}

\begin{proof}
Let $C$ be a clopen subset of $K$. By the definition of the adjoint operator, 
\[
P_C^*\phi=\phi\circ P_C,\qquad \phi\in Z^* .
\]
Since $C$ is clopen, both $\mathbf{1}_{C}$ and $\mathbf{1}_{K\setminus C}$ belong to $C(K)$. By the module property in Definition \ref{def:module}(iii), for every $z\in Z$, we have
\[
P_{C}z\coloneqq \mathbf{1}_{C}z\in Z,
\qquad
P_{K\setminus C}z\coloneqq\mathbf{1}_{K\setminus C}z\in Z.
\]
Thus every $z\in Z$ admits the unique decomposition $z=P_{C}z+P_{K\setminus C}z.$ Note that
\[
\begin{aligned}
\|z\|
&=\sup_{t\in K}\|z(t)\|=\max\left\{\sup_{t\in C}\|z(t)\|,\sup_{t\in K\setminus C}\|z(t)\|\right\} \\
&=\max\left\{\|P_{C}z\|,\|P_{K\setminus C}z\|\right\}.
\end{aligned}
\]
It then follows that
\begin{equation}\label{sum-infty}
Z=P_{C}Z\oplus_{\infty}P_{K\setminus C}Z.
\end{equation}
Taking dual spaces yields the canonical $\ell_{1}$-decomposition
\[
Z^{*}=\operatorname{Ran}P_{C}^{*}\oplus_{1}\operatorname{Ran}P_{K\setminus C}^{*}.
\]
Then $P_{C}^{*}$ is the $L$-projection of $Z^{*}$ onto $\operatorname{Ran}P_{C}^{*}=N_{C}^{\perp}$.
On the other hand, Gierz's construction associates to each clopen set $C$ an $L$-projection $p_{C}:Z^{*}\to Z^{*}$ with range precisely $N_{C}^{\perp}$.
The uniqueness of an $L$-projection with a prescribed range gives $p_{C}=P^{*}_{C}$.
Therefore, by the definition of the representing measure in
\eqref{Gierz-measure}, for every $\phi\in Z^{*}$,
\[
\mu_{\phi}(C)=\|p_{C}\phi\|=\|\phi\circ P_{C}\|.
\]
This completes the proof.
\end{proof}

\section{Canonical scalar measures and uniform integrability}\label{measure-integrability}

For the argument below, the key point is the precise identification of
$\mu_{\phi}$ on clopen sets. The resulting compactness argument can be
viewed as the function-module analogue of the compactness mechanism for
variation measures developed in \cite[Proposition~2]{Randrianantoanina1996}.

\begin{lem}\label{Canonical-measure compactness}
Let $K$ be a zero-dimensional compact Hausdorff space, and let $Z$ be a uniformly separable function module over $K$.
Suppose that $T:Z\to Y$ is unconditionally converging. Set
\[
\mathcal{H}=T^{*}(B_{Y^{*}})\subset Z^{*}, \quad \mathcal{M}=\{\mu_{\phi}:\phi\in\mathcal{H}\}\subset M(K),
\]
where $\mu_{\phi}$ denotes the canonical measure associated with $\phi\in Z^{*}$.
Then $\mathcal{M}$ is relatively weakly compact in $M(K)=C(K)^{*}.$
\end{lem}

\begin{proof}
According to \eqref{Gierz-measure}, for every $\phi\in\mathcal{H}$ we have $\mu_{\phi}(K)=\|\phi\|\leq \|T\|$, and hence $\mathcal{M}$ is norm bounded in $M(K)$. Assume that $\mathcal{M}$ is not relatively weakly compact. It follows from the Dieudonn\'e–Grothendieck criterion that there exist $\eta>0$ and a sequence of pairwise disjoint open sets $U_{j}\subset K$ such that 
\[
\mu_{\phi_{j}}(U_{j})>\eta,
\]
where $\phi_{j}=T^{*}y_{j}^{*}$ with $\|y_{j}^{*}\|\leq 1$. By the inner regularity of the measures $\mu_{\phi_j}$, for each $j$ there exists a compact set $F_j\subset U_j$ satisfying $\mu_{\phi_{j}}(F_{j})>\frac{\eta}{2}$. Since $K$ is zero-dimensional and compact, there exists a clopen set $C_{j}$ such that
\[
F_{j}\subset C_{j} \subset U_{j}.
\]
The sets $C_{j}$ are therefore pairwise disjoint, and
 $\mu_{\phi_{j}}(C_{j})>\frac{\eta}{2}$.
By Lemma \ref{lem clopen}, the canonical measure satisfies $\mu_{\phi_{j}}(C_{j})=\|\phi_{j}\circ P_{C_{j}}\|>\frac{\eta}{2}$. Consequently, for each $j$ there exists $u_{j}\in B_{Z}$ such that 
\[
\operatorname{supp} u_{j}\subset C_{j}, \qquad |\phi_{j}(u_{j})|>\frac{\eta}{4}.
\]
Now let $F\subset\mathbb{N}$ be finite and let $(\alpha_{j})_{j\in F}$ be scalars with $|\alpha_{j}|\leq1$. Disjointness of the supports gives
\[
\left\|\sum_{j\in F} \alpha_{j}u_{j}\right\| = \sup_{t\in K}\left|\sum_{j\in F} \alpha_{j}u_{j}(t)\right| \leq 1.
\]
Thus $\sum_{j} u_{j}$ is w.u.c.\ by \eqref{eq:wuc-test}. Since $T$ is unconditionally converging, $\sum_{j} Tu_{j}$ converges unconditionally, and in particular $\|Tu_{j}\| \to 0$. On the other hand, 
\[
\eta/4 < |\phi_{j}(u_{j})| = |y_{j}^{*}(Tu_{j})| \leq \|Tu_{j}\|,
\]
which is a contradiction.

\end{proof}

\begin{rem}
The topology used in Lemma \ref{Canonical-measure compactness} is the weak topology $\sigma(M(K),M(K)^{*}),$ not merely the weak-$*$ topology $\sigma(M(K),C(K)).$ This distinction is essential for the control-measure argument below.
\end{rem}

Fix a sequence $(\phi_{n})\subset\mathcal{H}$, and write $\mu_{n}=\mu_{\phi_{n}}.$ Define the control measure
\begin{equation}\label{control measure}
\lambda=\sum_{n=1}^{\infty}2^{-n}\mu_{n}.
\end{equation}
Then $\lambda$ is a finite positive regular Borel measure. If $\lambda=0$, then $\mu_{n}=0$ for every $n$. The identity $\|\phi_{n}\|=\mu_{n}(K)$ gives $\phi_{n}=0$ for all $n$ and the assertion is trivial. Assume therefore that
$\lambda\neq0$. By construction, $\mu_{n}\ll\lambda$ for every $n$ and we  denote the Radon--Nikodym derivative of $\mu_{n}$ with respect to $\lambda$ by
\begin{equation}\label{RN}
h_{n}=\frac{d\mu_{n}}{d\lambda}\in L_{1}(\lambda).
\end{equation}
The family $(h_{n})$ is uniformly integrable if and only if 
\begin{equation}\label{def: uniformly integrable}
\lim_{R\to \infty}\operatorname{sup}_{n}\int_{\{|h_{n}|>R\}}|h_{n}|d\lambda =0.
\end{equation}

\begin{prop}\label{thm: uniformly integrable}
The family $(h_{n})$ is relatively weakly compact in $L_{1}(\lambda)$ and hence uniformly integrable. In particular, 
\begin{equation}\label{truncation-uniform}
\lim_{R\to\infty}\sup_{n}\int_{K}(h_{n}-R)_{+}\,d\lambda=0.
\end{equation}
\end{prop}

\begin{proof}
Consider the band
\[
M_{\lambda}(K)\coloneq \{v\in M(K):\ v\ll\lambda\},
\]
which is norm closed in $M(K)$. Indeed, if $(v_{j})\subseteq M_{\lambda}(K)$ satisfying $\|v_{j}-v\|\to 0$, then $\lambda(A)=0$ implies $|v|(A)\le\|v-v_j\|$, and hence $v\ll\lambda$.
The Radon--Nikod\'{y}m correspondence $h\mapsto h\lambda$ is an onto linear isometry from $L_{1}(\lambda)$ to $M_{\lambda}(K)$. Since $M_{\lambda}(K)$ is a closed subspace, its weak topology is precisely the topology inherited from the weak topology of $M(K)$ by Hahn–Banach. Thus Lemma \ref{Canonical-measure compactness} implies the relative weak compactness of $(h_{n})$ in $L_{1}(\lambda)$. The Dunford--Pettis criterion gives uniform integrability and \eqref{truncation-uniform} follows immediately by \eqref{def: uniformly integrable}.
\end{proof}

\begin{rem}\label{rem-error}
Applying Theorem \ref{Gierz} to each $\phi_{n}$ and writing $\xi_{n}=\xi_{\phi_{n}}$, we have
\[
\phi_{n}(z)=\int_{K}\xi_{n}(t)(z(t))h_{n}(t)\,d\lambda(t).
\]
For $R>0$ let $k_{n,R}=h_{n}\wedge R$ and define 
\begin{equation}\label{truncation-function}
\phi_{n,R}(z)=\int_{K}\xi_{n}(t)(z(t))k_{n,R}(t)\,d\lambda(t).
\end{equation}
Since $\|\xi_{n}(t)\|\leq 1$, we have
\[
\|\phi_{n}-\phi_{n,R}\|\leq\int_{K}(h_{n}-R)_{+}\,d\lambda,
\]
and Proposition \ref{thm: uniformly integrable} implies that this error tends to zero uniformly in $n$ as $R\to\infty$.
\end{rem}

\section{Proof of Theorem \ref{thm:module}}\label{proof of Theorem 1.2}

This section is devoted to proving the function module theorem (i.e., Theorem \ref{thm:module}) over the real scalars. The complex case follows by the realiﬁcation argument in Section \ref{Introduction}. To this end, we need several results.

Let $Z$ be a real function module over $K$, and let $\lambda$ be a finite positive regular Borel measure. Define a seminorm
\begin{equation}\label{eq:envelope-norm}
q_{\lambda}(z)=\left(\int_{K}\|z(t)\|^{2}\,d\lambda(t)\right)^{1/2}, \quad z\in Z.
\end{equation}
Upper semicontinuity of the norm field gives that the integrand is Borel measurable. We denote by $H_{\lambda}$ the completion of $Z/\ker q_{\lambda}$ and by
\[
J_{\lambda}:Z\to H_{\lambda}
\]
the canonical map. Thus $\|J_{\lambda}z\|=q_{\lambda}(z)\le\lambda(K)^{1/2}\|z\|$ yields $J_{\lambda}$ is bounded. 
The next lemma is a common-modulus argument in the spirit of Day's theorem \cite{Day1941}, see also \cite[Theorem 3.5]{LucicPasqualettoVojnovic2024}.  We give the details for the completion just defined, thereby avoiding any identification with a separately constructed measurable-bundle space. 


\begin{lem}\label{lem:envelope-convexity}
Suppose that the stalks of $Z$ satisfy \eqref{eq:common-modulus}. Then 
\begin{equation}\label{eq:envelope-modulus}
\delta_{H_{\lambda}}(\varepsilon)
\ge1-\sqrt{1-\frac{\varepsilon^{2}}{256}\delta(\varepsilon/8)}
\qquad 0<\varepsilon\le2.
\end{equation}
\end{lem}

\begin{proof}
We first establish a midpoint estimate on the dense quotient $Z/\ker q_{\lambda}$. Choose representatives $u,v\in Z$ such that
\[
q_{\lambda}(u)=q_{\lambda}(v)=1,\qquad q_{\lambda}(u-v)\ge\varepsilon.
\]
For $t\in K$, set
\[
a=\|u(t)\|,\qquad b=\|v(t)\|,\qquad
m=\min(a,b),\qquad s=\frac{a+b}{2}.
\]
Then $\|a\|_{L^{2}}=\|b\|_{L^{2}}=1$. 
Suppose first that $\|a-b\|_{L^{2}}\ge\varepsilon/2$. The triangle inequality and the parallelogram identity in $L_{2}(\lambda)$ give
\begin{equation}\label{eq:radial-case}
q_{\lambda}\left(\frac{u+v}{2}\right)^{2}
\le\|s\|_{L^{2}}^{2}
=1-\frac{1}{4}\|a-b\|_{L^{2}}^{2}
\le1-\frac{\varepsilon^{2}}{16}.
\end{equation}
Assume now that $\|a-b\|_{L^{2}}<\varepsilon/2$. On the set where $a,b>0$, put
\[
\widehat{u}(t)=\frac{u(t)}{a(t)},\qquad
\widehat{v}(t)=\frac{v(t)}{b(t)},\qquad
\rho(t)=\|\widehat{u}(t)-\widehat{v}(t)\|;
\]
and set $\rho(t)=0$ where $m(t)=0$. We have
\[
\|u(t)-v(t)\|\le|a(t)-b(t)|+m(t)\rho(t)
\]
and hence $\|m\rho\|_{L^{2}}\ge\varepsilon/2$. Let
\[
A=\{t:\rho(t)\ge\varepsilon/4\}.
\]
Since $0\le\rho\le2$ and $\|m\|_{L^{2}}\le1$, we obtain
\[
\|\mathbf{1}_{K\setminus A}m\rho\|_{L^{2}}\le\varepsilon/4,
\qquad
\|\mathbf{1}_{A}m\|_{L^{2}}\ge\varepsilon/8.
\]
Set $d=\delta(\varepsilon/4)$. For $t\in A$, the common-modulus hypothesis applied in $Z_{t}$ yields
\begin{equation}\label{eq:pointwise-midpoint}
\left\|\frac{u(t)+v(t)}{2}\right\|\le s(t)-d\,m(t).
\end{equation}
Indeed, when $a\ge b$,
\[
\frac{u+v}{2}
=\frac{a-b}{2}\widehat{u}
+b\frac{\widehat{u}+\widehat{v}}{2};
\]
the case $b\ge a$ is symmetric. Outside $A$, the midpoint norm is bounded by $s$. Since $s\ge m$ and $0<d\le1$, $(s-dm)^{2}\le s^{2}-dm^{2}.$ Integrating the preceding estimates gives
\begin{equation}\label{eq:angular-case}
q_{\lambda}\left(\frac{u+v}{2}\right)^{2}
\le\|s\|_{L^{2}}^{2}-d\|\mathbf{1}_{A}m\|_{L^{2}}^{2}
\le1-\frac{\varepsilon^{2}}{64}\delta(\varepsilon/4).
\end{equation}
Since $\delta\le1$, the estimate \eqref{eq:radial-case} also implies the final bound in \eqref{eq:angular-case}. Thus this bound holds in both cases.

We briefly verify the measurability used above. The functions $a$ and $b$ are Borel. On the set where $a,b>0$, their reciprocal coefficients can be approximated pointwise by rational-valued Borel simple functions. The norm of every fixed rational linear combination of $u$ and $v$ is upper semicontinuous. Passing to pointwise limits shows that $\rho$ is measurable, and hence so is $A$.

Let now two unit vectors of $H_{\lambda}$ be at distance at least $\varepsilon$. Approximate them by normalized vectors from the dense quotient. The distances of the approximating vectors are eventually at least $\varepsilon/2$. Applying the estimate just proved with $\varepsilon/2$ and passing to the limit gives
\[
\left\|\frac{u+v}{2}\right\|^{2}
\le1-\frac{\varepsilon^{2}}{256}\delta(\varepsilon/8)
\]
for unit vectors $u,v\in H_{\lambda}$ at distance at least $\varepsilon$. This is precisely \eqref{eq:envelope-modulus}.
\end{proof}

\begin{prop}\label{relatively weakly compact}
For every fixed $R>0$, the set $\{\phi_{n,R} : n \in \mathbb{N}\}$ is relatively weakly compact in $Z^{*}$.
\end{prop}

\begin{proof}
Let $M=\sup_{n\in\mathbb{N}}\|\phi_{n}\|<\infty.$ Recall that $k_{n,R}=h_{n}\wedge R$. It follows that
\[
\|k_{n,R}\|_{L^{2}(\lambda)}^{2}\le R\int_{K} h_{n}\,d\lambda=R\mu_{n}(K)
=R\|\phi_{n}\|\le RM.
\]
Then by \eqref{truncation-function} and the Cauchy--Schwarz inequality,
\begin{equation}\label{truncation-bounded}
|\phi_{n,R}(z)|\le q_{\lambda}(z)\|k_{n,R}\|_{L^{2}(\lambda)}\le \sqrt{RM}\,q_{\lambda}(z),
\qquad z\in Z.
\end{equation}
In particular, $\phi_{n,R}$ vanishes on $\ker q_{\lambda}$. Thus the
formula
\[
g_{n,R}(J_{\lambda}z)=\phi_{n,R}(z),
\qquad z\in Z,
\]
defines a bounded linear functional on the dense subspace
$J_{\lambda}(Z)\subset H_{\lambda}$. The preceding estimate \eqref{truncation-bounded} gives
\[
\|g_{n,R}\|_{H_{\lambda}^{*}}\le \sqrt{RM}.
\]
Hence $g_{n,R}$ extends uniquely to $H_{\lambda}^{*}$, still
denoted by $g_{n,R}$, and
\[
\phi_{n,R}=J_{\lambda}^{*}g_{n,R}.
\]
By Lemma~\ref{lem:envelope-convexity}, $H_{\lambda}^{*}$ is reflexive. The set $\sqrt{RM}\,B_{H_{\lambda}^{*}}$ is consequently weakly compact. Since $J_{\lambda}^{*}:H_{\lambda}^{*}\to Z^{*}$
is weak-to-weak continuous, its image $J_{\lambda}^{*}\bigl(\sqrt{RM}\,B_{H_{\lambda}^{*}}\bigr)$ is weakly compact in $Z^{*}$. Note that $\{\phi_{n,R}:n\in\mathbb{N}\}\subset J_{\lambda}^{*} \bigl(\sqrt{RM}\,B_{H_{\lambda}^{*}}\bigr)$.
It follows that the family $\{\phi_{n,R}:n\in\mathbb{N}\}$ is relatively weakly compact in $Z^{*}$.
\end{proof}

We now turn to the proof of Theorem \ref{thm:module}.

\begin{proof}[Proof of Theorem~\ref{thm:module} over the real scalars]
Let $T:Z\to Y$ be an unconditionally converging operator. We prove that $T^{*}(B_{Y^{*}})$ is relatively weakly compact. By the Eberlein--\v{S}mulian theorem, it suffices to start with an arbitrary sequence
\[
(\phi_{n})\subset T^{*}(B_{Y^{*}})
\]
and produce a weakly convergent subsequence.

Recall that \eqref{control measure} provides a control measure $\lambda$. If $\lambda=0$, then every $\phi_{n}$ vanishes and the assertion is immediate. Suppose that $\lambda\ne0$. For every $j\ge1$, Remark \ref{rem-error} allows us to choose $R_{j}\uparrow \infty$ such that
\begin{equation}\label{eq:approximants-errors}
\sup_{n}\|\phi_{n}-\phi_{n,R_{j}}\|\le2^{-j}.
\end{equation}
For each fixed $j$, Proposition \ref{relatively weakly compact} gives that the sequence $(\phi_{n,R_{j}})_{n}$ is relatively weakly compact in $Z^{*}$. Repeatedly applying Eberlein--\v{S}mulian and then taking a diagonal subsequence, we obtain indices $(n_{k})$ and functionals $\psi_{j}\in Z^{*}$ such that
\begin{equation}\label{eq:weak-limits}
\phi_{n_{k},R_{j}}\rightharpoonup\psi_{j}
\quad\text{in }Z^{*}\quad \text{as}\quad k\to\infty
\end{equation}
for every fixed $j$.
The weak lower semicontinuity of the norm and \eqref{eq:approximants-errors} imply
\[
\|\psi_{j}-\psi_{l}\|\le2^{-j}+2^{-l},
\]
which yields that $(\psi_{j})$ is Cauchy in norm and then converges to some $\psi \in Z^{*}$. For $z^{**}\in Z^{**}$, estimates \eqref{eq:approximants-errors} and \eqref{eq:weak-limits} give
\[
|z^{**}(\phi_{n_{k}}-\psi)|\le\|z^{**}\|\,2^{-j}+|z^{**}(\phi_{n_{k},R_{j}}-\psi_{j})|+\|z^{**}\|\,\|\psi_{j}-\psi\|.
\]
Letting $k\to\infty$ and then $j\to\infty$, we obtain $\phi_{n_{k}}\rightharpoonup\psi$ in $Z^{*}$. 
We have proved that every sequence in $T^{*}(B_{Y^{*}})$ has a weakly convergent subsequence. Another application of Eberlein--\v{S}mulian shows that this set is relatively weakly compact. Hence $T^{*}$ is weakly compact, and Gantmacher's theorem implies that $T$ is weakly compact. This is property $(V)$.
\end{proof}

\section{Proof of Theorem \ref{thm:main}}\label{proof of Theorem 1.1}

We begin the proof of Theorem~\ref{thm:main} by addressing two preliminary issues before applying Theorem~\ref{thm:module} (i.e., Lemma \ref{lem:ultrapower-modulus} and Proposition \ref{stalk}). 
Let $E$ be uniformly convex, set $X=\ell_{\infty}(E)$, and identify $C(\beta\mathbb{N})$ with $\ell_{\infty}$ via the canonical isometric isomorphism. The module action is given by coordinatewise scalar multiplication.
For an ultrafilter $\mathcal{U}\in\beta\mathbb{N}$, define
\begin{equation*}\label{eq:ultrapower-definition}
N_{\mathcal{U}}=\left\{x=(x_{n})\in X:\lim_{\mathcal{U}}\|x_{n}\|=0\right\},
\qquad
E_{\mathcal{U}}=X/N_{\mathcal{U}}.
\end{equation*}
Since $N_{\mathcal{U}}$ is closed in $X$, the quotient  $E_{\mathcal{U}}$ is a Banach space.
We denote the image of $x\in X$ in this quotient by $x(\mathcal{U})=[x]_{\mathcal{U}}.$ When $\mathcal{U}$ is a principal ultrafilter, this quotient is naturally isometric to $E$, whereas for a free ultrafilter it is the usual Banach-space ultrapower of $E$. Although these facts are standard \cite{Heinrich1980}, we include the details required to verify the function-module structure. The quotient norm is given by
\begin{equation}\label{eq:ultrapower-norm}
\|x(\mathcal{U})\|=\lim_{\mathcal{U}}\|x_{n}\|.
\end{equation}
For each $x\in X$, the map $\mathcal{U}\mapsto\|x(\mathcal{U})\|$ is the continuous extension to $\beta\mathbb{N}$ of the bounded scalar sequence $(\|x_{n}\|)$. Moreover, since every principal ultrafilter belongs to $\beta\mathbb{N}$, $\sup_{\mathcal{U}\in\beta\mathbb{N}}\|x(\mathcal{U})\|=\sup_{n}\|x_{n}\|.$
The evaluation map is surjective by the definition of the quotient spaces.
Finally, if $f\in C(\beta\mathbb{N})$, then $(fx)(\mathcal{U})=f(\mathcal{U})x(\mathcal{U}).$
Therefore, the map $x\longmapsto(x(\mathcal{U}))_{\mathcal{U}\in\beta\mathbb{N}}$ realizes $X$ as a function module over $\beta\mathbb{N}$ with stalks $E_{\mathcal{U}}$ in the sense of Definition~\ref{def:module}.

\begin{lem}\label{lem:ultrapower-modulus}
If $E$ is uniformly convex, then for every ultrafilter $\mathcal{U}$,
\begin{equation}\label{eq:ultrapower-modulus}
\delta_{E_{\mathcal{U}}}(\varepsilon)\ge\delta_{E}(\varepsilon/2)>0
\qquad 0<\varepsilon\le2.
\end{equation}
Every closed subspace of $E_{\mathcal{U}}$ satisfies the same bound.
\end{lem}

\begin{proof}
There is nothing to prove when $E=\{0\}$. Otherwise, fix unit vectors
$u,v\in E_{\mathcal{U}}$ with $\|u-v\|\ge\varepsilon$, and choose bounded
representatives $(u_{n})$ and $(v_{n})$. Since
\[
\lim_{\mathcal{U}}\|u_{n}\|=\lim_{\mathcal{U}}\|v_{n}\|=1,
\]
we may normalize the coordinates, assigning an arbitrary fixed unit vector
at the zero coordinates, without changing either class in the quotient.
Formula~\eqref{eq:ultrapower-norm} then shows that
\[
\{n\in\mathbb{N}:\|u_{n}-v_{n}\|\ge\varepsilon/2\}\in\mathcal{U}.
\]
On this set, uniform convexity of $E$ gives
\[
\left\|\frac{u_{n}+v_{n}}{2}\right\|
\le1-\delta_{E}(\varepsilon/2).
\]
Taking the ultralimit proves \eqref{eq:ultrapower-modulus}. Finally, the definition of the modulus of convexity in \eqref{eq:modulus-definition} implies that every closed subspace of $E_{\mathcal{U}}$ satisfies the same bound.
\end{proof}

The module $X=\ell_{\infty}(E)$ need not be uniformly separable, even
when $E$ is separable. We therefore restrict the argument to a countably
generated submodule. Given a sequence $(x_{m})\subset X$, define
\begin{equation}\label{uniformly separable}
Z=\overline{\operatorname{span}}\{f x_{m}:f\in C(\beta\mathbb{N}),\ m\in\mathbb{N}\}\subset X
\end{equation}
to be the closed $C(\beta\mathbb{N})$-submodule generated by $(x_{m})$. For each $\mathcal{U}\in\beta\mathbb{N}$, let
\[
Z_{\mathcal{U}}=\{z(\mathcal{U}):z\in Z\}\subset E_{\mathcal{U}}
\]
be the evaluation image of $Z$ in $E_{\mathcal{U}}$.

\begin{prop}\label{stalk}
For each $\mathcal{U}\in\beta\mathbb{N}$, the evaluation image
\begin{equation}\label{eq:generated-fibre}
Z_{\mathcal{U}} = \overline{\operatorname{span}}\{x_{m}(\mathcal{U}) : m \in \mathbb{N}\}
\end{equation}
is a closed subspace of $E_{\mathcal U}$. Moreover, with these stalks $Z$ is a uniformly separable function module.
\end{prop} 

\begin{proof}
Fix $\mathcal{U}\in\beta\N$ and set $N_{\mathcal{U}}^{Z}=\{z\in Z:z(\mathcal{U})=0\}$. We first prove that evaluation induces an isometry from $Z/N_{\mathcal{U}}^{Z}$ to $E_{\mathcal{U}}$. Contractivity of evaluation gives
\[
\|z(\mathcal{U})\|\le\operatorname{dist}\left(z,N_{\mathcal{U}}^{Z}\right).
\]
Conversely, given $\varepsilon>0$, by the definition of ultralimit, there exists $A\in \mathcal{U}$ such that $\sup_{n\in A}\|z_{n}\|\le \|z(\mathcal{U})\|+\varepsilon$. The closure of $A$ in $\beta\mathbb{N}$ is clopen and $\mathbf{1}_{\bar{A}}\in C(\beta \mathbb{N})$. Hence $\mathbf{1}_{\bar{A}}z\in Z$, while $z-\mathbf{1}_{\bar{A}}z\in N_{\mathcal{U}}^{Z}$. It follows that
\[
\operatorname{dist}\left(z,N_{\mathcal{U}}^{Z}\right)\le\|\mathbf{1}_{\bar{A}}z\|\le \|z(\mathcal{U})\|+\varepsilon.
\]
After letting $\varepsilon\to0$, we conclude that the canonical map from $Z/N_{\mathcal{U}}^{Z}$ to $E_{\mathcal{U}}$ is an isometry onto $Z_{\mathcal{U}}$. Since the quotient is complete, its isometric image $Z_{\mathcal{U}}$ is closed in $E_{\mathcal{U}}$.
Suppose now that \eqref{uniformly separable} holds. Continuity of evaluation
and the identity
\[
(fx_m)(\mathcal U)=f(\mathcal U)x_m(\mathcal U)
\]
show that $Z_{\mathcal U}$ is contained in the right-hand side of
\eqref{eq:generated-fibre}. The reverse inclusion follows by using constant
scalar functions and the closedness of $Z_{\mathcal U}$. Finally, the rational linear span of $x_{m}$ is countable, and its evaluation images are dense in every $Z_{\mathcal{U}}$. Hence $Z$ is uniformly separable. Moreover, the norm fields of $Z$ are restrictions of the continuous norm fields of $X$, and therefore $Z$ is a function module.
\end{proof}

We are now ready to prove Theorem~\ref{thm:main}.

\begin{proof}[Proof of Theorem~\ref{thm:main}]
We first consider the real case. By the renorming theorem of Enflo and
Pisier, a super-reflexive space admits an equivalent uniformly convex norm
\cite{Enflo1972,Pisier1975}. The induced norms on $\ell_{\infty}(E)$ are
equivalent, while property $(V)$ and the Grothendieck property are
isomorphic invariants. We may therefore assume that the given norm on $E$
is uniformly convex.

Set $X = \ell_{\infty}(E)$. Suppose that an unconditionally converging operator $T : X \to Y$ is not weakly compact. By the Eberlein--\v{S}mulian theorem, there exists a bounded sequence $(x_{m}) \subset X$ such that $(Tx_{m})$ has no weakly convergent subsequence. Let $Z$ denote the closed $C(\beta\mathbb{N})$-submodule generated by $(x_{m})$ as in \eqref{uniformly separable}. Then the restriction $T|_{Z} : Z \to Y$ remains unconditionally converging and fails to be weakly compact.
By Proposition~\ref{stalk}, $Z$ is uniformly separable and its stalk at $\mathcal{U}$ is a closed subspace of $E_{\mathcal{U}}$. Moreover, Lemma~\ref{lem:ultrapower-modulus} ensures that all these stalks share the same positive modulus $\delta_{E}(\varepsilon/2)$. Since $\beta\mathbb{N}$ is zero-dimensional, Theorem~\ref{thm:module} implies that $Z$ has property $(V)$. Consequently, $T|_{Z}$ is weakly compact, which contradicts the construction. Hence $X$ has property $(V)$.

It remains to obtain the Grothendieck property. Since $E$ is super-reflexivity, we have 
\begin{equation}\label{eq:dual-identification}
\ell_{\infty}(E)=\ell_{\infty}(E^{**})\cong\bigl(\ell_{1}(E^{*})\bigr)^{*}.
\end{equation}
Hence $X$ is a dual Banach space. Apply Proposition~\ref{prop:dual-bridge} to conclude that $\ell_{\infty}(E)$ is Grothendieck. 

The complex case follows by the realiﬁcation argument in section \ref{Introduction}.
\end{proof}

We conclude this section with several operator-theoretic consequences of
property $(V)$. These consequences show that the combination of the geometry of $E$ and the structure of $\ell_{\infty}(E)$ imposes strong restrictions on operators defined on this space.

\begin{cor}\label{remaining}
Let $E$ be super-reflexive, $X=\ell_{\infty}(E)$, and $Y$ a Banach space. Then
\begin{enumerate}[\rm (i)]
\item Every non-weakly compact operator $T : X \to Y$ is an isomorphism on a subspace of $X$ isomorphic to $\ell_{\infty}$.

\item If $Y$ is separable, then every operator $T:X\to Y$ is weakly compact.

\item Every separable Banach quotient of $X$ is reflexive.

\item If $Y$ has the Schur property, then every operator $T:X\to Y$ is compact.
\end{enumerate}
\end{cor}

\begin{proof}
(i) follows immediately from Theorem~\ref{thm:main}, the dual representation \eqref{eq:dual-identification}, and Proposition~\ref{prop:dual-bridge}.
If $Y$ is separable, it cannot contain a subspace isomorphic to the nonseparable space $\ell_{\infty}$. Hence (ii) follows immediately from (i).

(iii) Let $T : X \to Y$ be a quotient map with $Y$ separable. By (ii), $T$ is weakly compact, and hence $T(B_{X})$ is relatively weakly compact in $Y$. Since $T$ is a surjective bounded linear operator, the open mapping theorem yields a constant $c > 0$ such that $B_{Y} \subset cT(B_{X}).$ On the other hand, $B_{Y}$ is weakly closed, since it is a norm-closed convex subset of $Y$. Hence $B_{Y}$ is weakly compact. It follows from Kakutani's theorem that $Y$ is reflexive. 

(iv) Let $Y$ have the Schur property. Since $Y$ contains no copy of $c_{0}$, the Bessaga--Pe\l czy\'{n}ski selection principle implies that every operator $T : X \to Y$ is unconditionally converging. By property (V), $T$ is weakly compact. Hence $T(B_{X})$ is relatively weakly compact. By the Eberlein--\v{S}mulian theorem, every sequence in $T(B_{X})$ has a weakly convergent subsequence, which is norm convergent by the Schur property. Therefore $T(B_{X})$ is relatively norm compact, which proves that $T$ is compact.

\end{proof}

\section*{Acknowledgement}

This work was supported by the National Natural Science Foundation of China (Grant Nos. 12125109, W2411005, 12671165 \& 125B2010 ); the Natural Science Foundation of Hunan Province (Grant Nos. 2025ZYJ002, 2024JJ1010 \& 2024RC3040);  the Science and Technology Innovation Program of Hunan Province (Grant No. 2024RC1016); and the Scientific Research Fund of Hunan Provincial Education Department (Grant Nos. 25A0009 \& 25B0008)

\end{document}